\documentclass[12pt]{amsart}

\usepackage{amssymb}
\usepackage{graphicx}

\newcommand{\eps}{\varepsilon}

\newcommand{\B}{{\mathcal B}}
\newcommand{\N}{{\mathbb N}}
\newcommand{\C}{{\mathbb C}}

\newcommand{\R}{{\mathbb R}}

\newcommand{\tef}{transcendental entire function}

\newcommand{\sconn}{simply connected}

\theoremstyle{plain}
\newtheorem{theorem}{Theorem}[section]
\newtheorem{corollary}[theorem]{Corollary}

\newtheorem{lemma}[theorem]{Lemma}

\theoremstyle{definition}

\theoremstyle{remark}

\theoremstyle{problem}

\theoremstyle{example}

\begin{document}


\title[Baker's conjecture]{Baker's conjecture for functions with real zeros}

\author{D. A. Nicks}
\address{School of Mathematical Sciences\\
The University of Nottingham\\
University Park\\
Nottingham NG7 2RD\\
UK}
\email{Dan.Nicks@nottingham.ac.uk}

\author{P. J. Rippon}
\address{Department of Mathematics and Statistics \\
The Open University \\
   Walton Hall\\
   Milton Keynes MK7 6AA\\
   UK}
\email{Phil.Rippon@open.ac.uk}

\author{G. M. Stallard}
\address{Department of Mathematics and Statistics \\
The Open University \\
   Walton Hall\\
   Milton Keynes MK7 6AA\\
   UK}
\email{Gwyneth.Stallard@open.ac.uk}

\thanks{The first author was supported by the EPSRC grant EP/L019841/1, and the last two authors were supported by the EPSRC grants EP/H006591/1 and EP/K031163/1.}




\begin{abstract}
Baker's conjecture states that a {\tef} of order less than $1/2$ has no unbounded Fatou components. It is known that, for such functions, there are no unbounded periodic Fatou components and so it remains to show that they can also have no unbounded wandering domains. Here we introduce completely new techniques to show that the conjecture holds in the case that the {\tef} is real with only real zeros, and we prove the much stronger result that such a function has no orbits of unbounded wandering domains whenever the order is less than~1. This raises the question as to whether such wandering domains can exist for \emph{any} \tef\;with order less than~1.

Key ingredients of our proofs are new results in classical complex analysis with wider applications. These new results concern: the winding properties of the images of certain curves proved using extremal length arguments, growth estimates for entire functions, and the distribution of the zeros of entire functions of order less than~$1$.
\end{abstract}

\maketitle

\section{Introduction}
\setcounter{equation}{0}
Let $f:\C\to \C$ be a {\tef} and denote by $f^{n},\,n=0,1,2,\ldots\,$, the $n$th iterate of~$f$. The {\it Fatou set} $F(f)$ is the set of points $z \in \C$ such that $(f^{n})_{n \in \N}$ forms a normal family in some neighborhood of $z$.  The complement of $F(f)$ is called the {\it Julia set} $J(f)$ of $f$. An introduction to the properties of these sets can be found in~\cite{wB93}.

Baker's conjecture, arising from his paper~\cite{iB81} in 1981, is that the Fatou set has no unbounded components whenever the order of the function is less than~$1/2$, or even whenever the function has order at most $1/2$, minimal type (see Section~2 for definitions). It is known~\cite{Z} that such functions have no unbounded periodic or preperiodic Fatou components but it remains open as to whether such a function can have an unbounded wandering domain, that is, an unbounded component $U$ of the Fatou set such that $f^n(U) \cap f^m(U) = \emptyset$ for $n \neq m$. Note that \cite{iB81} gives examples of functions of order $1/2$ with unbounded {\it periodic} Fatou components.

Many authors have shown that Baker's conjecture holds provided some regularity condition is imposed on the growth of the maximum modulus but, without any such condition, it is not even known whether the conjecture holds for all functions of order zero. The strongest results in this direction are given in~\cite{HM} and in~\cite{RS08}. A survey of earlier work on this conjecture appears in~\cite{H}. All these papers used properties of the minimum modulus to show that, under certain conditions, the images of certain curves must stretch radially.

The paper~\cite{RS13} introduced a new approach to the problem, showing that, for a certain class of functions, if these image curves do not stretch radially, then they must wind round the origin repeatedly. This led to the first result on Baker's conjecture for functions of positive order requiring no restriction on the regularity of the growth, namely, that the conjecture holds for a {\tef} of order less than~$1/2$ which is {\it real} (that is, it takes only real values on the real axis) and has only negative zeros.

This new approach to Baker's conjecture requires a detailed understanding of the influence of the zeros of the function on the images of curves. Even the apparently straightforward generalisation of allowing the zeros to lie anywhere on the real axis requires much more sophisticated arguments than were used in~\cite {RS13}, where repeated use was made of the fact that, when the zeros of $f$ are all negative, $|f(re^{i\theta})|$ is strictly decreasing for $0\le \theta\le \pi$, for any $r>0$.

In this paper we prove several new results in complex analysis which enable us to make this generalisation and, more surprisingly, enable us to show that such functions have no orbits of unbounded wandering domains whenever the order is less than 1.

\begin{theorem}\label{main}
Let~$f$ be a real {\tef} of order less than~1 with only real zeros. Then~$f$ has no orbits of unbounded wandering domains.
\end{theorem}

As stated earlier, it is known that functions of order at most 1/2, minimal type, have no unbounded periodic Fatou components. For such functions, the minimum modulus function is unbounded (see, for example,~\cite[p.274]{T}) and so the image of any unbounded continuum is also unbounded. Hence we have the following corollary to Theorem~\ref{main}.

\begin{corollary}\label{Baker}
Let $f$ be a real {\tef} of order at most 1/2, minimal type, with only real zeros. Then~$f$ has no unbounded Fatou components and hence Baker's conjecture holds.
\end{corollary}

The proofs of all earlier results on Baker's conjecture made crucial use of the very strong minimum modulus properties of functions of order less than 1/2, and the proof of Theorem~\ref{main} requires a number of new results concerning properties of functions of order less than 1. Note that there are examples of functions of order 1 with orbits of unbounded wandering domains -- for example, the function defined by $f(z) = z - 1 + e^{-z} + 2\pi i$ has this property as shown by Herman~\cite{mH84} -- but no examples are known of functions of order less than 1 which have this property. Our result suggests the following.

{\bf Question} \; Is there an example of a {\tef} $f$ of order less than~1 with an orbit of unbounded wandering domains?

This question adds a new perspective to the body of work seeking to identify classes of functions for which there are no wandering domains. This work began with the famous paper of Sullivan~\cite{dS82} which introduced the idea of quasiconformal deformations to the area and showed that rational functions have no wandering domains. Many subsequent papers show that Sullivan's techniques can be extended to a number of classes of {\tef}s and all papers ruling out wandering domains involve an analysis of the singular values of the function. Our result does not rule out bounded wandering domains, and indeed these can exist for functions of arbitrarily small growth; see~\cite{aH94}. It does, however, provide a new approach for ruling out orbits of unbounded wandering domains.

One of the most recent papers concerning the absence of wandering domains is~\cite{LH} by Mihaljevi\'c-Brandt and Rempe-Gillen, who prove that wandering domains do not exist for many functions in the Eremenko-Lyubich class $\B$ (consisting of {\tef}s for which the set of finite singular values is bounded). As an application, they show that there are no wandering domains of functions in class $\B$ which have order less than 1 and are real with all their zeros on the negative real axis. So, for example, the following families of functions have no wandering domains:
\[
f(z) = \lambda \cos \sqrt{az + b}, \quad f(z) = \lambda \frac{\sin \sqrt{az + b}}{\sqrt{az + b}},
\]
where $\lambda, a, b$ are real, with $a,\lambda\ne 0$.

Our result complements this by using a completely different approach to show, for example, that the following families of functions have no orbits of unbounded wandering domains:
\begin{equation}\label{examples}
f(z) = p(z) \cos \sqrt{az + b}, \quad f(z) = p(z) \frac{\sin \sqrt{az + b}}{\sqrt{az + b}},
\end{equation}
where $a, b$ are real, with $a\ne 0$, and $p(z)$ is a real polynomial with only real zeros. Such functions belong to class $\B$ only when $p(z)$ is a real constant. These classes include the function of order $1/2$ defined by $f(z) = (z/2) \cos \sqrt{z}$, which has an unbounded attracting invariant Fatou component but, by Theorem~\ref{main}, has no orbits of unbounded wandering domains.

\!{\it Remarks}\\
1. The examples in \eqref{examples} have only real zeros, all but finitely many of which lie on either the positive or the negative real axis. Nevertheless, these functions are not covered by the results in \cite{RS13}.

2. Theorem~\ref{main} does not exclude the possibility that~$f$ has an unbounded wandering domain whose forward orbit consists eventually of bounded Fatou components. However, the proof shows that for this class of functions such a wandering domain can occur only if~0 lies in the orbit of wandering domains. In fact it also shows that for these functions unbounded periodic Fatou components can only occur if~0 lies in the periodic cycle.

3. In~\cite{RS13} it was shown that the functions considered there have the property that the {\it escaping set}
\[
I(f)=\{z:f^n(z)\to\infty \mbox{ as } n\to\infty\}
\]
forms a connected set (with a structure called a spider's web) and so satisfies Eremenko's conjecture~\cite{E} that for any transcendental entire function~$f$ the components of $I(f)$ are all unbounded, in a strong way. The winding results obtained in this paper allow that result to be generalised, but such generalisations will be left until later work.

The structure of this paper is as follows. In Section~2, we state two new results, Theorems~\ref{main1} and~\ref{main2}, concerning the winding of image continua and deduce from these a further result, Theorem~\ref{main3}, which shows that the images of certain continua must stretch radially or wind round the origin. In Section~3 we deduce Theorem~\ref{main} from Theorem~\ref{main3}.

We prove Theorems~\ref{main1} and~\ref{main2} in Sections~5 and~7, respectively. For both these theorems we need several results that relate the behaviour of the maximum modulus and the minimum modulus of~$f$, which are deduced in Section~4 from a fundamental estimate of Beurling. Theorem~\ref{main1} is proved using an extremal length argument applied to $\log f$, which is univalent in the upper half-plane for the functions considered here and in fact for all functions in the Laguerre--P\'olya class. Theorem~\ref{main2} is proved using a new result concerning the location of the zeros of functions of order less than~1,  which may be of independent interest. This result is proved in Section~6 using estimates for the minimum modulus due to Cartwright. Note that Theorem~\ref{main2} is needed only for the case of functions of order at least~$1/2$ with $0\notin F(f)$.

\section{Winding of image curves}
\setcounter{equation}{0}
We begin this section by stating two new theorems which show that, for entire functions of the type covered by Theorem~\ref{main}, the images of certain continua~$\gamma$ must wind many times round the origin. These play a key role in our proof of Theorem~\ref{main}. We note that for our present purposes it would be sufficient to prove these results in the case when~$\gamma$ is a curve rather than a continuum, but we state these theorems for a continuum because of possible future applications.

We first introduce some notation. We define the {\it maximum} and {\it minimum modulus} of an entire function~$f$ by
\[
M(r)= M(r,f) = \max_{|z| = r}|f(z)|\, \mbox{ and }\, m(r)= m(r,f) = \min_{|z|=r} |f(z)|, \mbox{ for } r>0,
\]
and the order $\rho=\rho(f)$ of a {\tef} $f$ by
\[
\rho=\limsup_{r\to\infty}\frac{\log\log M(r,f)}{\log r}\,.
\]
Defining $\sigma$ by
\[
 \sigma =  \limsup_{r \to \infty}\, \frac{\log
 M(r)}{r^{\rho}},
\]
we say that the growth of $f$ is of {\it minimal type} if $\sigma = 0$,
{\it mean type} if $0 < \sigma < \infty$ and {\it maximal type}
if $ \sigma = \infty$.

We recall the following version of Hadamard convexity: for a {\tef} $f$, there exists a constant $R_0=R_0(f)>0$ such that
\begin{equation}\label{Had}
M(r^c)\ge M(r)^c, \;\;\text{for } r\ge R_0,\; c>1;
\end{equation}
see~\cite[Lemma~2.1]{RS08}. Throughout the paper $R_0$ denotes the constant in \eqref{Had}. We also recall the well-known property of any {\tef} that
\begin{equation}\label{maxmod}
\frac{\log M(r)}{\log r}\to \infty\quad\text{as }r \to \infty.
\end{equation}

Next, for $r>0$, we write $C(r) = \{z: |z| = r\}$ and, for $0<r_1<r_2$, we write
\[
A(r_1,r_2)=\{z:r_1<|z|<r_2\}\;\;\text{and}\;\;\overline A(r_1,r_2)=\{z:r_1\le |z|\le r_2\}.
\]

Finally we introduce notation associated with the winding of image curves. If $\gamma$ is a plane curve with an associated parametrisation and $\gamma$ meets no zeros of~$f$, then we denote the net change in the argument of $f(z)$ as $z$ traverses $\gamma$ by $\Delta\!\arg f(\gamma)$.

If $\gamma$ is a continuum (that is, a nontrivial, compact, connected subset of $\C$) having the property that all the zeros of~$f$ lie in the unbounded complementary component of $\gamma$, and if $z_0,z'_0$ is any pair of distinct points in~$\gamma$, then we denote the net change in the argument of $f(z)$ as $z$ traverses~$\gamma$ from $z_0$ to $z'_0$ by $\Delta\!\arg (f(\gamma);z_0,z'_0)$. This quantity is defined by choosing
\begin{itemize}
\item
any {\sconn} domain, $G$ say, that contains $\gamma$ but no zeros of $f$,
\item a branch, $g$ say, of $\log f$ in~$G$,
\end{itemize}
and putting
\begin{equation}\label{Deltadef}
\Delta\!\arg (f(\gamma);z_0,z'_0)=\Im (g(z'_0))-\Im (g(z_0)).
\end{equation}
Note that in the case when $\gamma$ is a simple arc with endpoints $z_0$ and $z'_0$ that meets no zeros of~$f$, then we have $\Delta\!\arg (f(\gamma);z_0,z'_0)=\Delta\!\arg (f(\gamma))$.

We now state our first result concerning the winding of the images of certain continua under a function~$f$ that satisfies the hypotheses of Theorem~\ref{main}. Roughly speaking, this result states that if~$|f|$ is neither too large nor too small on a continuum~$\gamma$ that lies in the upper half-plane and crosses a sufficiently thick annulus, then the image of~$\gamma$ must wind many times round~0.

\begin{theorem}\label{main1}
Let $f$ be a real transcendental entire function of order less than~2 with only real zeros, and let~$s$ and~$a$ be positive real numbers such that
\begin{equation}\label{s>}
s\ge R_0 \quad\text{and}\quad \log s \ge \frac{64}{a^2} + \frac{80\pi}{a}\,.
\end{equation}
If $\gamma$ is any continuum in $\{z:\Im z \ge 0\}$ that meets both $C(s)$ and $C(s^{1+a})$ with
\begin{equation}
1/M(s) \le |f(z)| \le M(s), \mbox{ for } z\in\gamma, \label{|f|<M}
\end{equation}
then there exist a continuum $\Gamma\subset\gamma\cap \overline{A}(s,s^{1+a})$ and $z_0,z'_0\in\Gamma$ such that
\[ \Delta\!\arg(f(\Gamma);z_0,z'_0) \ge \frac{1}{10\pi}\log M(s) \log s^a. \]
\end{theorem}
%

Theorem~\ref{main1} is a generalisation of a similar result proved in~\cite[Theorem 2.1]{RS13} in the case that all the zeros are negative and the continuum~$\gamma$ is a level curve of~$f$. The proof here is, however, very different to the one given in that paper since we can no longer assume that $|f(re^{i\theta})|$ is strictly decreasing for $0 \leq \theta \leq \pi$, for any $r>0$, and it uses extremal length together with an estimate of Beurling. We describe the estimate of Beurling and some of its consequences in Section~4, and recall the notion of extremal length and prove Theorem~\ref{main1} in Section~5.

Our next theorem concerning the winding of the images of certain continua under a function~$f$ satisfying the hypotheses of Theorem~\ref{main} is completely new and is required to deal with a situation that we do not need to consider for functions of order less than~$1/2$; see \cite[Theorem~2.2]{RS13}. Roughly speaking, the following theorem states that if~$|f|$ is not too large on a continuum~$\gamma$ that lies in the upper half-plane and crosses a sufficiently thick annulus, and~$f$ takes a very small value at a point on~$\gamma$ near the outer edge of the annulus, then the image of~$\gamma$ must wind many times round~0.

\begin{theorem}\label{main2}
Let $f$ be a real {\tef} of order $\rho <  1$ with only real zeros. There exist $K_1 = K_1(\rho) > 2$ and $R_1 = R_1(f) \geq R_0$ such that, if
\[
 L \geq K_1, \; s^{1/L} \geq R_1
\]
and $\gamma$ is any continuum in $\{z: \Im z \geq 0\}$ that meets both $C(s^{1/L})$ and $C(s)$ with
\begin{equation}\label{gamma-s}
  |f(z)| < M(s^{1/L}), \mbox{ for } z \in \gamma,
\end{equation}
and
\[
|f(z')| < \frac{1}{M(s^{1/L^{1/8}})}, \mbox{ for some } z' \in \gamma \cap \{z: s^{1/L^{1/8}} \leq |z| \leq s\},
\]
then there exist a continuum $\Gamma\subset \gamma\cap\overline A(s^{1/L},s)$ and $z_0, z'_0\in\Gamma$ such that
\[
\Delta\!\arg(f(\Gamma);z_0,z'_0) \geq  \log M(s^{1/L})\,.
\]
\end{theorem}

In order to prove Theorem~\ref{main2} we obtain the following result, which gives a fairly precise statement about the location of the zeros of an entire function of order less than~1 whose minimum modulus is relatively small on a long interval, and which may be of independent interest. Note that this result does not hold in the case $\alpha\ge 1$, as shown by the exponential function.

\begin{theorem}\label{main2a}
Let $f$ be a {\tef} with $f(0) = 1$. Suppose that, for some $R > e^{28}$, there exists $0 < \alpha < 1$ such that
\[
 \log M(r) \leq r^{\alpha}, \mbox{ for } r \geq 3R^{1/(1-\alpha)}.
\]
If
\[
\log m(r) \leq \frac12 \log M(r), \mbox{ for } r \in (R/4,R/2),
\]
then
\[
n(R^{1/(1-\alpha)}) - n(R/e^{28}) \geq \frac{\log M(R)}{28} - \frac{3}{1 - \alpha},
\]
where $n(r)$ denotes the number of zeros of~$f$ in $\{z: |z| \le r\}$, counted according to multiplicity.
\end{theorem}

We prove Theorem~\ref{main2a} in Section~6 and then Theorem~\ref{main2} in Section~7.

We conclude this section by using Theorems~\ref{main1} and~\ref{main2} to prove the following theorem showing that the images of certain continua must stretch radially or wind around the origin. This result plays a key role in the proof of Theorem~\ref{main}.

\begin{theorem}\label{main3}
Let~$f$ be a real {\tef} of order $\rho < 1$ with only real zeros. Let
\begin{equation}\label{condK2}
K_2=K_2(\rho)=\max\{K_1,2^{8}\},
\end{equation}
where $K_1=K_1(\rho)$ is the constant in Theorem~\ref{main2}, and~$R_2=R_2(f)$ satisfies
\begin{equation}\label{condR2}
R_2\ge \max\{R_1,\exp(320)\}\quad\text{and}\quad \log M(R_2) \ge 2\pi,
\end{equation}
where $R_1=R_1(f)$ is the constant in Theorem~\ref{main2}. If $L$ and $t$ satisfy
\begin{equation}\label{LR1}
L \geq K_2, \; t^{1/L} \ge R_2,
\end{equation}
$\gamma$ is any continuum in $\{z: \Im z \geq 0\}$ that lies in $\overline{A}(t^{1/L}, t)$ and meets both $C(t^{1/L})$ and $C(t)$, and
\[
 S = \max_{z \in \gamma} |f(z)|,
\]
then at least one of the following must hold:
\begin{itemize}
\item[(1)] $f(\gamma)$ meets both $C(S^{1/(L+2)})$ and $C(S)$, and $S \ge M(t^{1/L})$;

\item[(2)] $f(\gamma)$ meets both $C(S^{1/(L(1-\eps))})$ and $C(S)$, and
\[S = M(t^{1-\eps}), \mbox{ where } 0\le \eps \leq \frac{10}{\sqrt{\log t}}\,;\]

\item[(3)] there exist a continuum $\Gamma \subset \gamma$ and $z_0, z'_0\in\Gamma$ such that
\[
\Delta\!\arg(f(\Gamma);z_0,z'_0) \geq  \log M(t^{1/L}) \ge 2\pi.
\]
\end{itemize}
\end{theorem}
\begin{proof}
Suppose that $f$,~$L$ and~$t$ satisfy the hypotheses of the theorem.

We consider two separate cases.

(a) First, suppose that
\begin{equation}\label{S}
S = M(t^{1-\eps}), \mbox{ for some } \eps \in [0,1-1/L].
\end{equation}
Then case~(1) must hold if there exists $z \in \gamma$ with $|f(z)| \leq 1$, since
\[S=M(t^{1-\eps})\geq M(t^{1/L}) > 1,\]
by \eqref{condR2} and the fact that $t^{1/L}\ge R_2$.

If $|f(z)| > 1$ for all $z \in \gamma$, then there are two possibilities.

(i) If $\eps \leq 10/ \sqrt{\log t}$, then case~(2) holds since, by~\eqref{Had},
\[
M(t^{1/L}) \leq M(t^{1-\eps})^{1/(L(1-\eps))} = S^{1/(L(1-\eps))}.
\]

(ii) If $\eps > 10/ \sqrt{\log t}$, then we claim that we can apply Theorem~\ref{main1} to $\gamma$ with $s=t^{1-\eps}$ and $s^{1+a} = t$. Clearly~\eqref{|f|<M} holds. To show that~\eqref{s>} holds, we note that $a=\eps/(1-\eps)$ and so, by~\eqref{LR1}, \eqref{condK2} and \eqref{condR2},
\begin{eqnarray*}
\frac{64}{a^2} + \frac{80 \pi}{a}  & = & \frac{64 (1-\eps)^2}{\eps^2} + \frac{80 \pi (1-\eps)}{\eps}\\
& < & \left( \frac{64}{100}\log t + \frac{80 \pi}{10} \sqrt{\log t} \right) (1-\eps)\\
& = & \left( \frac{64}{100} + \frac{8 \pi}{\sqrt{\log t}} \right) (1-\eps) \log t\\
& < & (1-\eps)\log t = \log s,
\end{eqnarray*}
as required.

Thus it follows from Theorem~\ref{main1},~\eqref{LR1},~\eqref{condK2},~\eqref{condR2} and~\eqref{S}, together with the facts that $s^a=t^{\eps}$ and
\[\eps\log t > 10 \sqrt{\log t} > 10\pi,\]
that there exist a continuum $\Gamma \subset \gamma$ and $z_0, z_0' \in \Gamma$ such that
\[
\Delta\!\arg(f(\Gamma);z_0,z'_0) \geq \frac{\log M(t^{1-\eps}) \log (t^{\eps})}{10 \pi} >  \log M(t^{1/L}) \ge 2\pi.
\]
Thus case~(3) holds.

(b) Now suppose that~\eqref{S} does not hold and so
\begin{equation}\label{fsmall2}
|f(z)| < M(t^{1/L}), \mbox{ for } z \in \gamma.
\end{equation}
There are two possibilities.

(i) If
\[
|f(z)| \geq \frac{1}{M(t^{1/L^{1/8}})}, \mbox{ for } z \in \gamma \cap \{z: t^{1/L^{1/8}} \leq |z| \leq t\},
\]
then, by \eqref{fsmall2}, we can apply Theorem~\ref{main1} with $s = t^{1/L^{1/8}}$ and $s^{1+a} = t$, provided~\eqref{s>} holds. This does hold since $a = L^{1/8} - 1 \ge 1$, by \eqref{condK2}, and
\[
\log s = \log (t^{1/L^{1/8}}) > \log (t^{1/L}) > 320 > 64 + 80 \pi > \frac{64}{a^2} + \frac{80 \pi}{a},
\]
by \eqref{LR1} and \eqref{condR2}.

Thus, by Theorem~\ref{main1} and~\eqref{condR2}, there exist a continuum $\Gamma \subset \gamma$ and $z_0, z_0' \in \Gamma$ such that
\begin{eqnarray*}
\Delta\!\arg(f(\Gamma);z_0,z'_0) &\geq& \frac{\log M(s) \log s^a}{10 \pi}\\
&\geq& \frac{\log M(t^{1/L^{1/8}}) \log s}{10 \pi}\\
&\geq& \log M(t^{1/L}) \ge 2\pi.
\end{eqnarray*}
Thus case~(3) holds.

(ii) If
\[
|f(z')| < \frac{1}{M(t^{1/L^{1/8}})}, \mbox{ for some } z' \in \gamma \cap \{z: t^{1/L^{1/8}} \leq |z| \leq t\},
\]
then it follows from Theorem~\ref{main2}, with $s=t$, and~\eqref{condR2} that there exist a continuum $\Gamma \subset \gamma$ and $z_0, z_0' \in \Gamma$ such that
\[
\Delta\!\arg(f(\Gamma);z_0,z'_0) \geq  \log M(t^{1/L}) \ge 2\pi.
\]
Thus case~(3) holds.
\end{proof}

\section{Proof of Theorem~\ref{main}}
\setcounter{equation}{0}
In the proof of Theorem~\ref{main}, we use the following standard distortion theorem for iterates in escaping Fatou components; see \cite[Lemma~7]{wB93}.
\begin{lemma}\label{dist}
Let $f$ be a {\tef}, let $U \subset I(f)$ be a {\sconn} Fatou component of $f$, and let $K$ be a compact subset of $U$.
There exist $C>1$ and $N \in \N$ such that
\[
|f^n(z_0)| \leq C|f^n(z_1)|, \; \mbox{ for } z_0, z_1 \in K, \; n \geq N.
\]
\end{lemma}
We now give the proof of Theorem~\ref{main}.

\begin{proof}[Proof of Theorem~\ref{main}]
Let $f$ be a real {\tef} of order $\rho < 1$ with only real zeros and suppose that~$U$ is a wandering domain of~$f$ such that $U_n$ is unbounded for all $n \in \N$, where $U_n$ denotes the Fatou component containing $f^n(U)$. Since $U$ is wandering we may assume without loss of generality that
\begin{equation}\label{0U}
0 \notin U_n,\; \mbox{ for } n \in \N.
\end{equation}
(In fact this is the only point in the proof where we use the fact that $U$ is a wandering domain, so our proof also rules out the existence of unbounded periodic Fatou components for which all Fatou components in the cycle are unbounded and omit~0.) Note that the Fatou components $U_n$ are all simply connected since they are unbounded; see~\cite{iB75}.

We will show that the existence of such a wandering domain $U$ leads to a contradiction, thus proving Theorem~\ref{main}. The idea of the proof is as follows: we start with a curve $\gamma_0$ that lies in $U$ and we assume that the curve meets two circles of the form $C(r_0^{1/L_0})$ and $C(r_0)$, for some sufficiently large values of~$r_0$ and~$L_0$. We then obtain a contradiction by repeatedly applying Theorem~\ref{main3} to show that either the forward images of this curve must experience repeated radial stretching, and so contradict Lemma~\ref{dist} eventually, or $f^n(\gamma_0)$ must wind round~$0$, for some $n \in \N$, and hence $0 \in U_n$ contradicting the assumption \eqref{0U} about the forward images of~$U$.

A key fact needed to obtain these contradictions is that the real function $f$ has the following symmetry property:
\[
f(\overline z)=\overline{f(z)},\;\;\text{for } z\in\C.
\]
This property implies that $F(f)$ is symmetric with respect to the real axis.

Let $K_2=K_2(\rho)$ and $R_2=R_2(f)$ be as defined in Theorem~\ref{main3}. We begin by taking
\begin{equation}\label{mu}
L_0 \geq 2K_2 \geq 2^9, \quad \mu(r)=M(r^{1/L_0})
\end{equation}
and then
\begin{equation}\label{conds}
r_0 \geq R_2^{L_0} > \exp(320 L_0)
\end{equation}
such that, for $r \geq r_0$,
\begin{equation}\label{conds1}
\mu(r) > r^{16} \mbox{ and hence } \mu^n(r) \to \infty \mbox{ as } n \to \infty.
\end{equation}

Now suppose that $(L_n)$ and $(r_n)$ are sequences with the following properties: for each $n \ge 0$,
 either
\begin{equation}\label{Ln1}
L_{n+1} = L_0 \mbox{ and } r_{n+1} \geq M(r_n^{1/L_n})
\end{equation}
or
\begin{eqnarray}\label{Ln2}
&L_{n+1} = L_n(1 - \eps_n) \mbox{ and } r_{n+1} = M(r_n^{1-\eps_n}),\\[5pt]
&\mbox{where } 0 < \eps_n  \le \dfrac{10}{\sqrt{\log r_n}} \mbox{ and } m(r) \leq M(r_n^{1-\eps_n}), \mbox{ for } r_n^{1 - \eps_n} \leq r \leq r_n.\notag
\end{eqnarray}

Then we claim that, for each $n \geq 0$,

\begin{equation}\label{epsr}
\eps_n < \frac{1}{4^{n+1}} \mbox{ and } r_{n+1} \geq \mu(r_n).
\end{equation}
This is true for $n=0$ since, by~\eqref{mu} and~\eqref{conds},
\begin{equation}\label{eps0}
\eps_0 \le \dfrac{10}{\sqrt{\log r_0}} \leq \frac{10}{\sqrt{320L_0}} \le \frac{10}{\sqrt{320\times 2^9}} < \frac{1}{4}
\end{equation}
and hence
\[
1 - \eps_0 \geq \frac{3}{4} > \frac{1}{L_0}
\]
so that, by~\eqref{Ln1} and~\eqref{Ln2},
\[
r_1 \geq M\left(r_0^{1/L_0}\right) = \mu(r_0).
\]
Now suppose that
\[
\eps_k < \frac{1}{4^{k+1}}\quad \mbox{and}\quad r_{k+1} \geq \mu(r_k), \mbox{ for } 0 \leq k \leq n.
\]
Then, by~\eqref{conds1}, $r_{n+1} \geq \mu^n(r_0) > r_0^{16^n}$ and so, by~\eqref{Ln2},~\eqref{mu} and \eqref{conds},
\[
\eps_{n+1} \le \frac{10}{\sqrt{\log r_{n+1}}} < \frac{10}{4^n \sqrt{\log r_0}} \le \frac{10}{4^n\sqrt{320\times 2^9}} < \frac{1}{4^{n+2}}.
\]
So, by~\eqref{Ln1} and~\eqref{Ln2}, $r_{n+2} \geq \mu(r_{n+1})$. Thus~\eqref{epsr} follows by induction.

We note that it follows from~\eqref{epsr},~\eqref{Ln1} and~\eqref{Ln2} that, for $n \geq 0$,
\[
L_0 \geq L_n \geq L_0\prod_{m=0}^{n-1}(1 - \eps_m) \geq L_0\prod_{m=0}^{n-1}\left(1- \frac{1}{4^{m+1}}\right) \geq \frac{1}{2}L_0,
\]
and so
\begin{equation}\label{Ln4}
L_0/2 \leq L_n \leq L_0, \mbox{ for } n \geq 0.
\end{equation}

Now let $\gamma_0$ be the curve described earlier and suppose that there exist curves $\gamma_n\subset f^n(\gamma_0)$ such that, for $n \geq 0$, we have
\begin{equation}\label{contains}
f(\gamma_{n})\supset \gamma_{n+1},
\end{equation}
\begin{equation}\label{eq3.2}
\gamma_n\subset \overline A(r_n^{1/L_n},r_n),\;\text{ and }\gamma_n\text{ meets both }C(r_n^{1/L_n}) \text{ and } C(r_n).
\end{equation}
(These conditions formalise what we mean by saying that the images of the curve~$\gamma_0$ experience `repeated radial stretching'.)

We deduce, by \eqref{contains}, that there is a point $z\in \gamma_0$ such that, for $n\ge 0$,
\begin{equation}\label{gamman}
f^n(z)\in \gamma_n,\quad \text{so}\quad |f^n(z)|\ge r_n^{1/L_n} \geq (\mu^n(r_0))^{1/L_0},
\end{equation}
by~\eqref{mu},~\eqref{Ln4} and~\eqref{epsr}. Together with~\eqref{conds1}, this implies that $U\subset I(f)$. Since it also follows from~\eqref{Ln4},~\eqref{epsr},~\eqref{contains} and~\eqref{eq3.2} that
\begin{align*}
\frac{\sup \{|f^n(z)|:z\in\gamma_0\}}{\inf \{|f^n(z)|:z\in\gamma_0\}}&\ge \frac{r_n}{r_n^{1/L_n}}\ge \frac{r_n}{r_n^{2/L_0}}=r_n^{1- 2/L_0}\\
&\geq \mu^n(r_0)^{1 - 2/L_0} > \mu^n(r_0)^{1/2} \to\infty\;\;\text{as }n\to\infty,
\end{align*}
we have a contradiction to Lemma~\ref{dist}.

To construct the sequences~$(L_n)$,~$(r_n)$ and~$(\gamma_n)$ above, we proceed as follows. Take $L_0$ and $r_0$ to satisfy~\eqref{conds} and~\eqref{conds1}, and $\gamma_0$ to satisfy~\eqref{eq3.2}. Suppose that, for $k=1,\ldots, n$, we have chosen curves $\gamma_k$ and positive numbers $r_k$ and $L_k$ such that~\eqref{eq3.2} is satisfied with $n$ replaced by $k$ and, in addition, either~\eqref{Ln1} or~\eqref{Ln2} and also~\eqref{contains} hold, with $n$ replaced by $k$ for $k=0,1,\ldots,n-1$. To complete the proof we show that we can choose a curve $\gamma_{n+1}$ and positive numbers $r_{n+1}$ and $L_{n+1}$ so that~\eqref{contains} and also either~\eqref{Ln1} or~\eqref{Ln2} hold, and~\eqref{eq3.2} holds with $n$ replaced by $n+1$.

To do this we apply Theorem~\ref{main3}, with
\[
t=r_n \mbox{ and } L=L_n,
\]
to a curve $\gamma'_n$ meeting $C(r_n^{1/L_n})$ and $C(r_n)$, chosen such that $\gamma'_n \subset \{z: \Im z \geq 0\}$ and
\begin{equation}\label{gamma}
\gamma'_n \subset \gamma_n \cup \gamma_n^*,
 \end{equation}
where $*$ denotes reflection in the real axis.

Note that the hypotheses of Theorem~\ref{main3} are satisfied since $L_0 \geq L_n \geq L_0/2 \geq K_2(\rho)$, by~\eqref{Ln4} and~\eqref{mu}, and
\[
r_n^{1/L_n} \geq \mu^n(r_0)^{1/L_n} \geq r_0^{1/L_0} \geq R_2,
\]
by~\eqref{epsr},~\eqref{conds1},~\eqref{Ln4} and~\eqref{conds}.

If case~(1) or case~(2) of Theorem~\ref{main3} holds for $\gamma'_n$, then the same case holds for~$\gamma_n$, by the symmetry of $f$ in the real axis, and so we can choose $r_{n+1}$, $L_{n+1}$ and $\gamma_{n+1} \subset f(\gamma_n)$ so that they satisfy either~\eqref{Ln1} or~\eqref{Ln2}, \eqref{contains} and also~\eqref{eq3.2} with~$n$ replaced by $n+1$, as required.

Thus, to complete the proof it is sufficient to show that if case~(3) of Theorem~\ref{main3} holds for~$\gamma'_n$, then we obtain a contradiction.

If case~(3) holds for~$\gamma'_n$, then the image under~$f$ of some subcurve of $\gamma'_n$  winds round~$0$ through an angle of at least $2\pi$. Hence, by the symmetry of $f$ in the real axis, the Fatou component $U_{n+1}$ that contains $f^{n+1}(U)$ also contains a Jordan curve that surrounds~0. Since~$U_{n+1}$ is simply connected, it must contain~$0$. This, however, contradicts our assumption~\eqref{0U}, so the proof is complete.
\end{proof}

\section{Applications of a result of Beurling}
\setcounter{equation}{0}

Many authors have studied the relationship between the maximum modulus and the minimum modulus of a {\tef}, in particular the fact that, in some sense, if the minimum modulus is small, then the maximum modulus is forced to have large growth; see \cite[Chapter~8]{wH89} for many such results.

In our proofs of Theorems~\ref{main1} and~\ref{main2} we need new results of this type, which are consequences of the following result from Beurling's thesis \cite[page~96]{Beu}. For any subset $E$ of $(0,\infty)$ we denote by $m_{\ell}(E)$ the logarithmic measure of $E$:
\[
m_\ell(E)=\int_E \frac{dt}{t}.
\]
\begin{lemma}\label{Beur}
Let $f$ be analytic in $\{z:|z| < r_0\}$, let $0\le r_1<r_2< r_0$, and put
\[
E=\{t\in (r_1,r_2):m(t)\le \mu\}, \;\;\text{where } 0<\mu<M(r_1).
\]
Then
\begin{equation}\label{estimate}
\log \frac{M(r_2)}{\mu}>\frac{1}{2}\exp\left(\frac12 m_{\ell}(E)\right)\log \frac{M(r_1)}{\mu}.
\end{equation}
\end{lemma}

We remark that Beurling proved this estimate with the first constant on the right-hand side equal to $\pi/(4\sqrt 2)$, but for simplicity we use the value~$1/2$ here. We also remark that further applications of Lemma~\ref{Beur} were given in~\cite{RS13b} and~\cite{RS13a}.

Our first application here of Lemma~\ref{Beur} is an estimate for the growth of the maximum modulus over an interval on which the minimum modulus is less than the maximum modulus to a fixed power less than 1.

\begin{lemma}\label{Be3}
Let $f$ be a {\tef} and suppose that
\[
\log m(t)\le (1-\delta)\log M(t),\;\;\text{for }r\le t\le r^k,
\]
where $r>1$, $k\ge 1+4\log 2/\log r$ and $0 < \delta < 1$. Then
\[
\log M(r^k) > r^{\delta\log k/(16\log 2)}\log M(r).
\]
\end{lemma}
\begin{proof}
We begin by putting $\kappa=1+4\log 2/\log r$ and taking $n$ to be the largest integer such that
\[
\kappa^n\le k.
\]
Note that, by hypothesis, $n\ge 1$.

Then, for $0 \leq m \leq n$, we put $r_m=r^{\kappa^m}$. For $0 \leq m <n$, we apply Lemma~\ref{Beur} to the interval $(r_m,r_{m+1})$ with $\mu = M(r_{m+1})^{1-\delta}$ to deduce that
\begin{eqnarray*}
\log M(r_{m+1}) & > & \left(2\delta r_m^{-(\kappa-1)/2}+1-\delta\right)^{-1}\log M(r_m)\\
& = & \left(2\delta r_m^{-\log 4/\log r}+1-\delta\right)^{-1}\log M(r_m)\\
& \geq & \left(2\delta r^{-\log 4/\log r}+1-\delta\right)^{-1}\log M(r_m)\\
& = & \frac{1}{1-\delta/2}\log M(r_m).
\end{eqnarray*}

Hence
\begin{equation}\label{M-ineq1}
\log M(r^k)\ge \log M(r_n)> \left(\frac{1}{1-\delta/2}\right)^n \log M(r).
\end{equation}

Now,
\[
\kappa^{n+1}>k,
\]
so
\begin{equation}\label{M-ineq2}
n\ge \frac{n+1}{2}>\frac{\log k}{2\log \kappa}=\frac{\log k}{2\log (1+4\log 2/\log r)}\ge \frac{\log k\log r}{8\log 2}.
\end{equation}
Thus, by~\eqref{M-ineq1} and~\eqref{M-ineq2},
\begin{eqnarray*}
\log M(r^k)&\ge& \left(\frac{1}{1-\delta/2}\right)^{\log k\log r/(8\log 2)}\log M(r)\\
&=& r^{-\log(1-\delta/2)\log k/(8\log 2)}\log M(r)\\
&>&r^{\delta\log k/(16\log 2)}\log M(r),
\end{eqnarray*}
as required.
\end{proof}

Next we prove a {\it local} version of a $\cos \pi \rho$-type result. (For a description of classical $\cos \pi \rho$-type results, see~\cite{wH89}.) In~\cite{RS13b} we used Beurling's estimate in Lemma~\ref{Beur} to obtain a local result of this type for functions of order less than $1/2$. Here we need a result that can be applied to functions of order less than 1. In fact, Lemma~\ref{Be1} can be applied to functions of {\it any} order, and shows that the minimum modulus cannot be too small everywhere on an interval in relation to the value of the maximum modulus at the upper end of the interval. This result is not sharp but it is sufficient for our purposes.

\begin{lemma}\label{Be1}
Let $f$ be a {\tef}, $0 < \delta \leq 1$ and
\[0 < \lambda \leq \left( \frac{\delta}{2(1 + \delta)}\right)^2.\]
If $M(\lambda r) \geq 1$, then
\[
m(t) > \frac{1}{M(r)^{\delta}}, \mbox{ for some } t \in (\lambda r, r).
\]
\end{lemma}
\begin{proof}
If the conclusion is false, then it follows from Lemma~\ref{Beur} with $\mu = 1/M(r)^{\delta}$ that
\[
\log M(r)^{1+\delta} > \frac{1}{2} \exp\left( \frac{1}{2} \int_{\lambda r}^r \frac{dt}{t}\right) \log M(r)^{\delta}.
\]
This implies that
\[
 (1 + \delta) \log M(r) > \frac{1}{2} \sqrt{\frac{1}{\lambda}}\, \delta \log M(r),
\]
which in turn implies that $\lambda > \left( \frac{\delta}{2(1 + \delta)}\right)^2$. This, however, is a contradiction.
\end{proof}

\section{Proof of Theorem~\ref{main1}}
\setcounter{equation}{0}

The proofs of Theorems~\ref{main1} and~\ref{main2} both use the notion of extremal length and we begin this section by summarising some of the key results about extremal length that we use. For more details see, for example,~\cite{Ahl}.

Let $\Omega \subset \C$ be a domain and let $\Gamma$ be a collection of rectifiable arcs in $\Omega$. For a Riemannian metric $\rho|dz|$ on $\Omega$, each $\gamma \in \Gamma$ has a well-defined length
\[
L(\gamma,\rho) = \int_{\gamma} \rho\, |dz|
\]
and $\Omega$ has a well-defined area
\[
A(\Omega,\rho) = \int\!\! \int_{\Omega} \rho^2 dx \, dy.
\]
We put
\[
L(\Gamma, \rho) = \inf_{\gamma \in \Gamma} L(\gamma, \rho)
\]
and define the {\it extremal length} of $\Gamma$ in $\Omega$ to be
\[
\lambda_{\Omega}(\Gamma) =  \sup_{\rho} \frac{L(\Gamma,\rho)^2}{A(\Omega,\rho)}\,,
\]
where the $\sup$ is taken over all $\rho$ such that $0 < A(\Omega, \rho) < \infty$.

The following key results about extremal length can be found in~\cite[pages 50--53]{Ahl}.

\begin{lemma}\label{ext0}
Let $\Omega \subset \C$ be a domain and let $\Gamma$ be a collection of rectifiable arcs in $\Omega$. If $F$ is a conformal map on $\Omega$, then $\lambda_{\Omega}(\Gamma) = \lambda_{F(\Omega)}(F(\Gamma))$.
\end{lemma}

\begin{lemma}\label{ext1}
Let $\Omega \subset \C$ be a domain and let $\Gamma$ and $\Gamma'$ be collections of rectifiable arcs in $\Omega$. If every $\gamma\in \Gamma$ has a subarc $\gamma'\in \Gamma'$, then $\lambda_{\Omega}(\Gamma) \ge \lambda_{\Omega}(\Gamma')$.
\end{lemma}
These two lemmas can be used to give the following distortion theorem for conformal mappings, which is used in the proof of Theorem~\ref{main2}. Roughly speaking, this result states that if a quadrilateral is `long and thin' in a certain sense related to one choice of opposite sides, then its image under a conformal mapping cannot be long and thin with respect to the other choice of opposite sides.  The result is of independent interest and may be known, but we are not aware of a reference.

\begin{lemma}\label{long-thin}
Let $Q$ be a quadrilateral, that is, a Jordan domain together with four boundary points that divide $\partial Q$ into two pairs of opposite sides $\alpha, \alpha'$ and $\beta, \beta'$. Let $\phi$ be conformal on $\overline Q$ and let $a$, $b$, $A$ and $B$ be positive constants such that
\begin{equation}\label{tt1}
|\Re z' -\Re z| \ge a,\quad\text{for } z\in \alpha,\; z'\in \alpha',
\end{equation}
\begin{equation}\label{tt2}
|\Im z' -\Im z| \le b,\quad\text{for } z\in \beta,\; z'\in \beta',\; \Re z=\Re z',
\end{equation}
\begin{equation}\label{tt3}
|\Re w' -\Re w| \le A,\quad\text{for } w\in \phi(\alpha),\; w'\in \phi(\alpha'),\; \Im w=\Im w',
\end{equation}
\begin{equation}\label{tt4}
|\Im w' -\Im w| \ge B,\quad\text{for } w\in \phi(\beta),\; w'\in \phi(\beta').
\end{equation}
Then
\[
\frac ab\le \frac AB.
\]
\end{lemma}
\begin{proof} By an approximation argument, we can assume that $\partial Q$ is a piecewise analytic curve. We let $I$ denote a union of vertical crosscuts $I_x$ of $Q$, with real part $x$ and height $h(x)$ that separate $\alpha$ from $\alpha'$, which can be chosen, by \eqref{tt1}, so that~$h$ is defined and measurable on an interval~$J$ of length $\tilde a\ge a$ (see \cite[p. 93]{Nev}). Then let $\rho = \chi_I$ be the characteristic function of~$I$.

If $\Gamma$ is the collection of arcs in $Q$ that join $\alpha$ to $\alpha'$, then
\[
L(\Gamma,\rho) = \inf_{\gamma \in \Gamma} L(\gamma,\rho) \geq \tilde a
\]
and
\[
A(Q,\rho) = \int\!\!\int_{I} \rho^2\, dx \,dy = \int_{J} h(x)\, dx.
\]
Since $h(x)\le b$ for $x\in J$, by \eqref{tt2}, we deduce that
\begin{equation}\label{extr1}
\lambda_Q(\Gamma) \geq \frac{L(\Gamma,\rho)^2}{A(Q,\rho)} \ge \frac{\tilde a^2}{\tilde ab}=\frac{\tilde a}{b}\ge \frac{a}{b}.
\end{equation}
Now let $\Gamma'$ denote the collection of arcs in $Q$ that join $\beta$ to $\beta'$. Then (see \cite[page~53]{Ahl}),
\begin{equation}\label{extr2}
\lambda_Q(\Gamma')=1/\lambda_Q(\Gamma).
\end{equation}

Now $\phi(\Gamma')$ is the collection of arcs in $\phi(Q)$ that join $\phi(\beta)$ to $\phi(\beta')$, and
\begin{equation}\label{extr3}
\lambda_Q(\Gamma')=\lambda_{\phi(Q)}(\phi(\Gamma')),
\end{equation}
since $\phi$ is a conformal map. By using a similar argument to the above, involving horizontal crosscuts of $\phi(Q)$, we deduce from \eqref{tt3} and \eqref{tt4} that
\begin{equation}\label{extr4}
\lambda_{\phi(Q)}(\phi(\Gamma'))\ge \frac BA.
\end{equation}
Combining \eqref{extr1}, \eqref{extr2}, \eqref{extr3} and \eqref{extr4} gives $a/b\le A/B$ as required.
\end{proof}
%

The proofs of Theorems~\ref{main1} and~\ref{main2} also use the result that, for the functions~$f$ being considered, $\log f$ is conformal in the upper half-plane. Actually this result holds more generally for all functions in the so-called Laguerre--P\'olya class, which consists of all entire functions that can be approximated uniformly on compact subsets of the plane by real polynomials with all zeros
real. This class includes all real functions of order less than~2 with only real zeros; see \cite[page~266]{T}. The conformality of $\log f$ follows from a well-known monotonicity property of such functions (see, for example, \cite[Lemma~2.2 and Proposition~4.1]{CS00}) together with the Noshiro--Warschawski theorem (see \cite[page~46]{pD83}) but we include the proof for completeness.

\begin{lemma}\label{incr}
Let $f$ be a {\tef} in the Laguerre--P\'olya class. Then
\begin{itemize}
\item[(a)]
for $x \in \R$, $|f(x+iy)|$ is strictly increasing with respect to $y$, for $y>0$,
\item[(b)]
any analytic branch of $\log f$ is conformal in the upper half-plane.
\end{itemize}
\end{lemma}
\begin{proof}
It was proved by P\'olya \cite{gP} that every function in the Laguerre--P\'olya class is of the form
\[
f(z)=cz^{p_0}e^{az+bz^2}\prod_{j\in \N}\left(1+\frac{z}{a_j}\right)^{p_j}e^{-p_j z/a_j},
\]
where $c\in\R\setminus\{0\}$, $a\in \R$, $b\le 0$, $a_j\in \R\setminus\{0\}$ for $j \in \N$, and $p_j\in \{0,1,\ldots\}$ for $j\geq 0$.

Let $g$ be an analytic branch of $\log f$ in the upper half-plane $\mathbb H=\{z:\Im z>0\}$, which exists since~$f$ has no zeros in $\mathbb H$. For $z\in \mathbb H$,
\begin{equation}\label{Img'}
\Im (g'(z))=\Im\left(\frac{f'(z)}{f(z)}\right)=\Im\left(\frac{p_0}{z}+a+2bz+\sum_{j\in\N}\left(\frac{p_j}{a_j+z}-\frac{p_j}{a_j}\right)\right)<0.
\end{equation}
Hence, for each $y>0$,
\[
\Im (g(x+iy))=\arg f(x+iy)\; \text{is strictly decreasing with } x, \;\text{for } x\in\R,
\]
so, by the Cauchy-Riemann equations, for each $x\in \R$,
\[
\Re (g(x+iy))=\log |f(x+iy)|\; \text{is strictly increasing with } y, \;\text{for } y>0.
\]
This proves part~(a).

We now deduce that~$g$ is one-one in $\mathbb H$. For distinct $z_1$ and $z_2$ in $\mathbb H$, let $\gamma(t) = (1-t)z_1+t z_2,\;t\in [0,1]$. Then
\begin{eqnarray*}
g(z_2)-g(z_1)&=&\int_{z_1}^{z_2}g'(z)\,dz= \int_0^1g'(\gamma(t))(z_2-z_1)\,dt\\
&=&(z_2-z_1)\left(\int_0^1\Re (g'(\gamma(t))\,dt+i \int_0^1\Im (g'(\gamma(t))\,dt\right).
\end{eqnarray*}
By~\eqref{Img'}, we deduce that $g(z_2)\ne g(z_1)$, as required for part~(b).
\end{proof}

We now give the proof of Theorem~\ref{main1}.

\begin{proof}[Proof of Theorem~\ref{main1}]
We show first that
\begin{equation}
\log M(s^{1+a/2}) > 2\log M(s). \label{2logM}
\end{equation}
In order to show this, we first note that it follows from \eqref{s>} and \eqref{Had} that
\[ \log M(s^{1+a/4}) \ge (1+a/4)\log M(s). \]
Hence, by Lemma~\ref{Beur}, with $r_1=s^{1+a/4}$, $r_2=s^{1+a/2}$ and $\mu=M(s)$,
\begin{equation}
\log M(s^{1+a/2}) > \frac12 s^{a/8}\log \frac{M(s^{1+a/4})}{M(s)} + \log M(s) \ge \left(\frac{a}{8}s^{a/8} +1\right)\log M(s). \label{Beurling}
\end{equation}
It follows from \eqref{s>} that
\[ \frac{a}{8}s^{a/8} = \frac{a}{8}\exp\left(\frac{a}{8}\log s\right) \ge \frac{a}{8}\left(\frac{a}{8}\log s\right) \ge 1, \]
and this together with \eqref{Beurling} implies \eqref{2logM}.

Let $G_1$ denote the component of $\{z: |f(z)|<M(s^{1+a/2})\}$ such that $\gamma\subset G_1$. Since $\gamma$ joins $C(s)$ to $C(s^{1+a})$, the component $G_1$ contains $\{z:|z|<s^{1+a/2}\}$. Hence $G_1$ is symmetric with respect to the real axis, and $\partial G_1$ meets $C(s^{1+a/2})$ and $C(s^{1+a})$. Let $\Gamma \subset \overline{A}(s^{1+a/2},s^{1+a})$ be a subcontinuum of $\gamma$ that meets $C(s^{1+a/2})$ and $C(s^{1+a})$; such a subcontinuum exists by \cite[Lemma~3.3]{RS13}, for example.

Let
\[ \Omega = A(s^{1+a/2},s^{1+a}) \cap \{z: \Im z >0\} \quad \mbox{and} \quad \beta=\partial G_1\cap \overline{\Omega}, \]
and denote by $\Delta$ the collection of all rectifiable arcs in~$\Omega$ that join $\Gamma$ to~$\beta$, and by $\Delta_0$ the collection of all rectifiable arcs in $\Omega$ that join the real interval $[-s^{1+a},-s^{1+a/2}]$ to $[s^{1+a/2},s^{1+a}]$. Then, by the symmetry property of $G_1$, every arc in $\Delta_0$ has a subarc in $\Delta$ and hence, by Lemma~\ref{ext1},
\begin{equation}
\lambda_\Omega(\Delta) \le \lambda_\Omega(\Delta_0) = \frac{2\pi}{a \log s}. \label{upperbound}
\end{equation}
Let $F$ be an analytic branch of $\log f$ on the upper half-plane. Since~$f$ is a real {\tef} of order less than~2 with only real zeros, and hence in the Laguerre--P\'olya class, we deduce by Lemma~\ref{incr} part~(b) that the branch~$F$ is conformal on the upper half-plane. Note that $F$ extends continuously to $\Gamma\cup\beta$ (this may include some points on $\R$, but $f(z)\ne0$ at such points). Thus every element in $F(\Delta)$ is an arc in $F(\Omega)$ that joins $F(\Gamma)$ to $F(\beta)$. Moreover, by Lemma~\ref{ext0},
\[ \lambda_\Omega(\Delta) = \lambda_{F(\Omega)}(F(\Delta)). \]
We next find a lower bound for this extremal length.

Take $z_0, z'_0 \in\Gamma$ such that
\begin{equation}
\theta_0 := \inf_{z\in\Gamma} \Im (F(z)) = \Im (F(z_0)), \quad \theta'_0 := \sup_{z\in\Gamma} \Im (F(z)) = \Im (F(z'_0)) \label{theta}
\end{equation}
and write $\theta=\theta'_0 - \theta_0 = \Delta\!\arg(f(\Gamma);z_0,z'_0)$.
Let $R$ be the rectangle
\[ R= \left\{u+iv : |u|<\log M(s^{1+a/2}), \ \theta_0-\log\frac{M(s^{1+a/2})}{M(s)} < v < \theta'_0+\log\frac{M(s^{1+a/2})}{M(s)} \right\} \]
and let $\rho = \chi_R$ be the characteristic function of~$R$. Now
\begin{align*}
A(F(\Omega),\rho) &= \int\!\!\int_{F(\Omega)}\rho^2\,dxdy \\
&\le \mbox{Area}(R)\\
&= 2\log M(s^{1+a/2})\left(\theta+2\log\frac{M(s^{1+a/2})}{M(s)}\right).
\end{align*}
Since $\beta\subset\partial G_1$ we have that $F(\beta)\subset\{u+iv: u = \log M(s^{1+a/2})\}$. From \eqref{|f|<M} and \eqref{theta} we obtain \[F(\Gamma)\subset\{u+iv : |u|<\log M(s), \ \theta_0\le v \le\theta'_0\}.\]
We deduce that
\[ L(F(\Delta),\rho) = \inf_{\sigma\in F(\Delta)} \int_\sigma \chi_R\,|dz| \ge \log\frac{M(s^{1+a/2})}{M(s)}\,,  \]
because every arc in $F(\Delta)$ joins $F(\Gamma)$ to $F(\beta)$ and hence must have a subarc in~$R$ of length at least $\log(M(s^{1+a/2})/M(s))$.

Therefore
\[ \lambda_{F(\Omega)}(F(\Delta)) \ge \frac{L(F(\Delta),\rho)^2}{A(F(\Omega),\rho)} \ge \frac{\left(\log\frac{M(s^{1+a/2})}{M(s)}\right)^2}{2\log M(s^{1+a/2})\left(\theta+2\log\frac{M(s^{1+a/2})}{M(s)}\right)}. \]
Together with \eqref{upperbound} and the conformal invariance $\lambda_\Omega(\Delta) = \lambda_{F(\Omega)}(F(\Delta))$, this yields
\[ \theta \ge \frac{a\log s \left(\log\frac{M(s^{1+a/2})}{M(s)}\right)^2}{4\pi \log M(s^{1+a/2})} - 2\log\frac{M(s^{1+a/2})}{M(s)}. \]
Using \eqref{2logM} and \eqref{s>} now gives
\begin{align*}
\theta &\ge \frac{a\log s \left(\frac12 \log M(s^{1+a/2})\right)^2}{4\pi \log M(s^{1+a/2})} - 2\log M(s^{1+a/2}) + 2\log M(s) \\
& = \left(\frac{a\log s}{16\pi} - 2\right)\log M(s^{1+a/2}) + 2\log M(s) \\
&\ge \left(\frac{a\log s}{8\pi} - 4\right)\log M(s) + 2\log M(s) \\
&= \left(\frac{a\log s}{10\pi} + \frac{a\log s}{40\pi} - 2\right)\log M(s) \ge \frac{1}{10\pi}\log M(s) \log s^a,
\end{align*}
as required.
\end{proof}

\section{Proof of Theorem~\ref{main2a}}
\setcounter{equation}{0}

The remaining sections of this paper give results that are only needed to deal with real functions with real zeros that have order in the interval $[1/2,1)$.

Let $f$ be a {\tef} of order less than 1 with $f(0) = 1$. In the proof of Theorem~\ref{main2a}, we use some standard results about the following quantities:
\[ N(r) = \int_0^r \frac{n(t)}{t}\,dt\quad\text{and}\quad Q(r) = r \int_r^{\infty} \frac{n(t)}{t^2}\, dt,\]
where $n(r)$ is the number of zeros of $f$ in $\{z: |z| \leq r\}$, counted according to multiplicity.

We note that it follows from Jensen's theorem (see, for example,
\cite[p.125]{T} for a proof) that, for $r>0$,
\begin{equation}\label{J}
\log M(r) \geq N(r).
\end{equation}
Therefore, for $r>0$,
\begin{equation}\label{Mn}
\log M(r) \geq \int_{r/e^{28}}^r \frac{n(r/e^{28})}{t}\,dt = 28 n(r/e^{28})
\end{equation}
and
\begin{equation}\label{nM}
 n(r) \log 3 \leq \int_r^{3r} \frac{n(t)}{t}\,dt \leq N(3r) \leq \log
 M(3r).
\end{equation}

We also use the following results proved in~\cite[Lemma 3.3 and Lemma 3.5]{RS08}. Note that similar results were proved by Cartwright~\cite[page~83]{C} for functions of order zero.

\begin{lemma}\label{L33}
Let $f$ be a transcendental entire function of order less than~$1$
with $f(0) = 1$. Then, for $r > 0$,
\[
 \log M(r) \leq N(r) + Q(r).
\]
\end{lemma}

\begin{lemma}\label{L35}
Let $f$ be a transcendental entire function of order less than $1$
with $f(0) = 1$, and let $0 < \eta < 1/4$. Then, for $R>0$,
\[
  \log m(r) > N(R) - (1 + \log(2e/\eta)) Q(R),
\]
for $0 \leq r \leq R/2$ except in a set of intervals, the sum of
whose lengths is at most $2\eta R$.
\end{lemma}

In~\cite[Lemma 3.5]{RS08} we stated Lemma~\ref{L35} with the condition that $R$ is sufficiently large but inspection of the proof shows that the result in fact holds for all $R>0$. Taking $\eta=1/16$, we obtain the following corollary of Lemma~\ref{L35}.

\begin{corollary}\label{min}
Let $f$ be a transcendental entire function of order less than $1$
with $f(0) = 1$. Then, for $R>0$,
\[
\log m(r) > N(R) - 6Q(R)
\]
for $0 \leq r \leq R/2$ except in a set of intervals, the sum of whose lengths is at most $R/8$.
\end{corollary}

We are now in a position to prove Theorem~\ref{main2a}.

\begin{proof}[Proof of Theorem~\ref{main2a}]
Let $f$ be a {\tef} with $f(0) = 1$ and suppose that, for some $R > e^{28}$, there exists $0 < \alpha < 1$ such that
\begin{equation}\label{alpha}
 \log M(r) \leq r^{\alpha}, \mbox{ for } r \geq 3R^{1/(1-\alpha)}.
\end{equation}
Recall that Theorem~\ref{main2a} states that, if
\begin{equation}\label{mM}
\log m(r) \leq \frac12\log M(r), \mbox{ for } r \in (R/4,R/2),
\end{equation}
then
\[
n(R^{1/(1-\alpha)}) - n(R/e^{28}) \geq \frac{\log M(R)}{28} - \frac{3}{1 - \alpha}.
\]
In order to prove this, we begin by noting that it follows from~\eqref{alpha},~\eqref{mM} and Corollary~\ref{min} that
\[
\frac{1}{2} \log M(R) \geq N(R) - 6Q(R)
\]
and so
\[
\log M(R) \geq 2N(R) - 12Q(R).
\]
Together with Lemma~\ref{L33}, applied with $r = R$, this implies that
\begin{equation}\label{QM}
Q(R) \geq \frac{1}{14} \log M(R).
\end{equation}

We now note that it follows from~\eqref{nM} together with~\eqref{alpha} that
\begin{eqnarray*}
Q(R) & = & R\int_R^{R^{1/(1-\alpha)}}\frac{n(t)}{t^2}\,dt + R \int_{R^{1/(1-\alpha)}}^{\infty}\frac{n(t)}{t^2}\,dt\\
& \leq & n(R^{1/(1-\alpha)}) + 3^{\alpha} R \int_{R^{1/(1-\alpha)}}^{\infty}\frac{1}{t^{2-\alpha}}\,dt\\
& = & n(R^{1/(1-\alpha)}) + 3^{\alpha} R \left[  \frac{-1}{(1-\alpha)t^{1-\alpha}}\right]_{R^{1/(1-\alpha)}}^{\infty}\\
& \leq & n(R^{1/(1-\alpha)}) + \frac{3}{1-\alpha}.
\end{eqnarray*}

Together with~\eqref{QM} and~\eqref{Mn}, this implies that
\begin{eqnarray*}
n(R^{1/(1-\alpha)}) + \frac{3}{1-\alpha}  \geq  Q(R) & \geq & \frac{1}{28} \log M(R) + \frac{1}{28} \log M(R)\\
& \geq & \frac{1}{28} \log M(R) + n(R/e^{28})
\end{eqnarray*}
and so
\[
n(R^{1/(1-\alpha)}) - n(R/e^{28}) \geq \frac{\log M(R)}{28} - \frac{3}{1 - \alpha},
\]
as claimed.
\end{proof}

\section{Proof of Theorem~\ref{main2}}
\setcounter{equation}{0}
We begin by proving the following result. Recall that $R_0$ denotes the constant in~\eqref{Had}.

\begin{lemma}\label{zeros}
Let $f$ be a real {\tef} of order $\rho <  1$ with only real zeros.
There exist $K_1 = K_1(\rho) > 1$ and $R_1 = R_1(f) \geq R_0$ such that, if
\begin{equation}\label{LR}
 L \geq K_1, \; s^{1/L} \geq R_1,
 \end{equation}
 and $\gamma$ is any continuum in $\{z: \Im z \geq 0\}$ that meets both $C(s^{1/L})$ and $C(s)$ with
\begin{equation}\label{fsmalla}
  |f(z)| < M(s^{1/L}), \mbox{ for } z \in \gamma,
\end{equation}
and
\begin{equation}\label{vsmall}
|f(z')| < \frac{1}{M(s^{1/L^{1/8}})}, \mbox{ for some } z' \in \gamma \cap \{z: s^{1/L^{1/8}} \leq |z| \leq s\},
\end{equation}
then either\\
{\rm (a)} we have
\begin{itemize}
\item $\gamma$ meets $\{z: \Re z = s^{1/L}\}$ and $\{z: \Re z = s^{1/L^{1/8}}/16\}$,
\item there exists $s'' \in (s^{1/L^{1/4}},\Re z')$ with $m(s'') > 1/M(s^{1/L^{1/8}})$,
\item the number of zeros of $f$ in $I = [s^{1/L^{3/4}},s^{1/L^{1/4}}]$ is at least $\log M(s^{1/L^{3/4}})$,
\end{itemize}
or\\
{\rm (b)} we have
\begin{itemize}
\item $\gamma$ meets $\{z: \Re z = -s^{1/L}\}$ and $\{z: \Re z = -s^{1/L^{1/8}}/16\}$,
\item there exists $s'' \in (s^{1/L^{1/4}},-\Re z')$ with $m(s'') > 1/M(s^{1/L^{1/8}})$,
\item the number of zeros of $f$ in $I = [-s^{1/L^{1/4}}, -s^{1/L^{3/4}}]$ is at least $\log M(s^{1/L^{3/4}})$.
\end{itemize}
\end{lemma}
\begin{proof}
We first put
\begin{equation}\label{K1cond}
K_1=K_1(\rho)=\max\{\exp(200),1/(1-\alpha)^4\}, \;\;\text{where }\alpha=\frac12(\rho+1),
\end{equation}
and assume throughout the proof that $L\ge K_1$.

We shall choose $R_1=R_1(f)\ge R_0$ so that the various conditions required during the proof hold whenever $s^{1/L} \ge R_1(f)$.

It follows from Lemma~\ref{Be1}, with $\delta = 1$, $\lambda = 1/16$ and $r = s^{1/L^{1/8}}$, that, if $R_1$ is chosen to ensure that, whenever $s^{1/L}\ge R_1$,
\begin{equation}\label{M16}
M(s^{1/L}) \geq 1 \mbox{ and } s^{1/L^{1/8}}/16 > s^{1/L^{1/4}},
\end{equation}
then there exists
\begin{equation}\label{sm}
s'' \in (s^{1/L^{1/8}}/16,s^{1/L^{1/8}}) \subset (s^{1/L^{1/4}}, s^{1/L^{1/8}}) \mbox{ such that } m(s'') > \frac{1}{M(s^{1/L^{1/8}})}.
\end{equation}
It then follows from Lemma~\ref{incr} that we have
\[
|f(z)| > \frac{1}{M(s^{1/L^{1/8}})}, \mbox{ for } |z| \geq s'', \; |\Re z| \leq s''.
\]
Thus the point $z' \in \gamma$ described in~\eqref{vsmall} lies outside $\{z: |\Re z| \leq s''\}$. We assume that $\Re z' > s''$. (With this assumption we end up with case~(a) of the lemma, as the first two bullet points of case~(a) are now satisfied -- otherwise, similar arguments would lead to case~(b).)

Since~$\gamma$ meets both $C(s^{1/L})$ and $\{z:\Re z \geq s''\}$, and $s'' \geq s^{1/L^{1/4}}$, it follows from Lemma~\ref{incr} and~\eqref{fsmalla} that
\begin{equation}\label{upper}
|f(x)| \leq M(s^{1/L}), \mbox{ for } s^{1/L} \leq x \leq s^{1/L^{1/4}}.
\end{equation}
Since~$f$ is real with only real zeros and has order less than 1, we can write
\[
f(z) = cz^{p_0}\prod_{j\in \N}(1 + z/a_j)^{p_j},
\]
where $c \in \R \setminus \{0\}$, $p_0 \in \{0,1,\ldots\}$, $a_j \in \R\setminus \{0\}$ and $p_j \in \{1,2\ldots\}$, for $j \geq 1$, with $0 < |a_1| \leq |a_2| \leq \ldots$.
We now take $s^{1/L^{3/4}} > |a_1|$ and put
\[I = [s^{1/L^{3/4}},s^{1/L^{1/4}}].\]
We write $f(z) = g(z)h(z)$, where
\[
 h(z) = cz^{p_0}(1+z/a_1)^{p_1}\prod_{a_j \in I}(1 + z/a_j)^{p_j}.
\]
Our aim is to apply Theorem~\ref{main2a} to~$g$ in order to show that $f$ has many zeros in~$I$. To do this, we must check that the conditions of Theorem~\ref{main2a} hold. Clearly $g(0)=1$. We first require $s^{1/L}$ to be so large that
\begin{equation}\label{tL}
 s^{1/L} \geq R_0=R_0(f) \quad\mbox{and} \quad 4 \leq M(r,h) \leq 2 |h(r)|, \mbox{ for } r \in I.
\end{equation}
It follows from~\eqref{upper},~\eqref{tL} and~\eqref{Had} that, for $r \in I$, we have
\begin{eqnarray*}
|g(r)| = \frac{|f(r)|}{|h(r)|} & \leq & \frac{M(s^{1/L},f)}{|h(s^{1/L^{3/4}})|} \leq  \frac{M(s^{1/L^{3/4}},f)^{1/2}}{|h(s^{1/L^{3/4}})|}\\
 & \leq & \frac{M(s^{1/L^{3/4}},g)^{1/2}M(s^{1/L^{3/4}},h)^{1/2}}{|h(s^{1/L^{3/4}})|}\\
 & \leq & \frac{2M(s^{1/L^{3/4}},g)^{1/2}M(s^{1/L^{3/4}},h)^{1/2}}{M(s^{1/L^{3/4}},h)}\\
 & < & M(s^{1/L^{3/4}},g)^{1/2} \leq M(r,g)^{1/2}.
\end{eqnarray*}
Thus
\begin{equation}\label{mg}
\log m(r,g) < \frac12\log M(r,g), \mbox{ for } r \in I,
\end{equation}
so the second hypothesis of Theorem~\ref{main2a} is satisfied whenever $(R/4,R/2)\subset I$.

We next recall that $f$ has order $\rho < 1$ and $\alpha =\tfrac12(\rho+1)$, so there exists $r_1 = r_1(f) > 1$ such that
\begin{equation}\label{r1}
\log M(r,f) \leq r^{\alpha} \mbox{ and } |cz^{p_0}(1+z/a_1)^{p_1}| \geq 1, \mbox{ for }  |z| = r \geq r_1.
\end{equation}
We define $R>0$ by $R^{1/(1-\alpha)} = s^{1/L^{1/4}}$ and note that
\begin{equation}\label{tL0}
R = s^{(1- \alpha)/L^{1/4}} \ge s^{1/L^{1/2}}, \;\text{ for } s>0,
\end{equation}
by \eqref{K1cond}, and also impose the following further conditions on~$s^{1/L}$:
\begin{equation}\label{tL1}
s^{1/L^{3/4}} \geq r_1 \mbox{ and } R/e^{28} = s^{(1- \alpha)/L^{1/4}}/ e^{28} > s^{1/L^{3/4}}.
\end{equation}
We note that, if $a_j \in I$ and $|z| \geq 3R^{1/(1-\alpha)} = 3s^{1/L^{1/4}}$, then $|1 + z/a_j| \geq 2$. So it follows from~\eqref{r1} and~\eqref{tL1} that
\[
|h(z)| \geq 1, \mbox{ for } |z| \geq 3R^{1/(1-\alpha)}.
\]
Thus, for $r \geq 3R^{1/(1-\alpha)}$, it follows from~\eqref{r1} and~\eqref{tL1} that
\[
 \log M(r,g) \leq \log M(r,f) \leq r^{\alpha}.
\]
Together with~\eqref{mg} and~\eqref{tL1}, this implies that the hypotheses of Theorem~\ref{main2a} are satisfied for $g$ with $R = s^{(1 - \alpha)/L^{1/4}}$ and so
\[
n(R^{1/(1-\alpha)}) - n(R/e^{28}) \geq \frac{\log M(R,g)}{28} - \frac{3}{1 - \alpha},
\]
where $n(r)$ denotes the number of zeros of $g$ in $\{z: |z| \le r\}$, counted according to multiplicity. Clearly zeros of $g$ are also zeros of $f$ and it follows from the definition of $R$ and~\eqref{tL1} that zeros of $g$ of modulus between $R/e^{28}$ and $R^{1/(1-\alpha)}$ must lie in $I=[s^{1/L^{3/4}}, s^{1/L^{1/4}}]$.

To complete the proof of Lemma~\ref{zeros} it remains to show that
\begin{equation}\label{Mgf}
\frac{\log M(R,g)}{28} - \frac{3}{1 - \alpha} \geq \log M(s^{1/L^{3/4}},f).
\end{equation}
To do this, we first prove that if $r_2=r_2(f)=|a_2|e$, then $\log M(s^{1/L^{3/4}},g) > 1$, provided that $s^{1/L}\ge r_2$. Indeed, in this case, we have $s^{1/L^{3/4}} >|a_2|$, so $a_2\notin I$ and
\begin{eqnarray*}
M(r,g) &\ge& |g(ir)|\\
&=&\left|1+ir/a_2\right|^{p_2}\prod_{j\ge 3,\,a_j\notin I}\left|1+ir/a_j\right|^{p_j}\\
&\ge& \left|1+ir/a_2\right|^{p_2} \ge r/|a_2|,
\end{eqnarray*}
and hence
\begin{equation}\label{Mg1}
\log M(s^{1/L^{3/4}},g)\ge \log\left(s^{1/L^{3/4}}/|a_2|\right) >1,
\end{equation}
as required.

We now note that it follows from~\eqref{tL1},~\eqref{mg} and Lemma~\ref{Be3} (with $\delta=1/2$, $r=s^{1/L^{3/4}}$ and $r^k=s^{1/L^{1/2}}$), together with~\eqref{tL0},~\eqref{Mg1},~\eqref{tL} and~\eqref{K1cond}, that
\begin{eqnarray*}
\log M(R,g) & = & \log M(s^{(1-\alpha)/L^{1/4}},g) \ge \log M(s^{1/L^{1/2}},g) \\
& \ge & \left(s^{1/L^{3/4}}\right)^{(\log L^{1/4}/ 32 \log 2)} \log M(s^{1/L^{3/4}},g)\\
& > &  s^{\log L/(100L^{3/4})}\\
& \ge & s^{2/L^{3/4}}.
\end{eqnarray*}

Finally we note that it follows from~\eqref{r1} and~\eqref{tL1} that
\[
 \log M(s^{1/L^{3/4}},f) < s^{1/L^{3/4}}
\]
and so~\eqref{Mgf} holds provided that $s^{1/L}$ is so large that we also have
\[
\frac{1}{28}  s^{2/L^{3/4}} - \frac{3}{1- \alpha} \geq s^{1/L^{3/4}}. \qedhere
\]
\end{proof}

Finally we use Lemma~\ref{zeros} to prove Theorem~\ref{main2}.

\begin{proof}[Proof of Theorem~\ref{main2}]
Recall that Theorem~\ref{main2} states that if the function~$f$, the continuum~$\gamma$ and the point~$z'$ satisfy the conditions of Lemma~\ref{zeros}, with
\begin{equation}\label{sLR0}
L\ge K_1 = K_1(\rho) > 1\;\text{ and }\; s^{1/L}\ge R_1 = R_1(f) \geq R_0,
\end{equation}
then there exist a continuum $\Gamma\subset \gamma\cap\overline A(s^{1/L},s)$ and $z_0, z'_0\in\Gamma$ such that
\begin{equation}\label{arg}
\Delta\!\arg(f(\Gamma);z_0,z'_0) \geq  \log M(s^{1/L})\,.
\end{equation}
Without loss of generality we can assume that the continuum $\gamma$ lies entirely in $\{z:\Im z>0\}$.

We assume that case~(a) of Lemma~\ref{zeros} holds and let $s''$ denote the point described there; the proof if case~(b) holds is similar. We also take the constants $K_1=K_1(\rho)$ and $R_1=R_1(f)$ to satisfy the conditions in Lemma~\ref{zeros}, and also one further condition below.

In order to prove the existence of a continuum $\Gamma$ satisfying~\eqref{arg}, we consider the pre-images under $f$ of the positive real axis that lie in the upper half-plane and originate at the zeros of $f$ in the interval $I$ defined in Lemma~\ref{zeros} case~(a). By Lemma~\ref{zeros}, there are at least $\log M(s^{1/L^{3/4}})$ such zeros, and hence at least $\tfrac12\log M(s^{1/L^{3/4}})$ corresponding pre-images of the positive real axis. We will show that at least $\tfrac14\log M(s^{1/L^{3/4}})$ of these pre-images must meet $\gamma\cap\overline A(s^{1/L},s)$.

Consider a simply connected domain $U_I$ that is bounded by appropriate parts of~$\gamma$, the positive real axis, the line $C_1 = \{z: \Re z = \Re z'\}$ and the circle $C(s^{1/L})$. We note that $I\subset \partial U_I$. Also, the pre-images defined above are unbounded and so each has a first point in $\{z:\Im z>0\}$ at which it crosses the boundary of $U_I$. It is sufficient to show that at least $\tfrac14\log M(s^{1/L^{3/4}})$ of these points belong to~$\gamma$.

We begin by noting that these first crossing points cannot lie in $C_1$ since each of the pre-images would have to cross $\{z: |z| = s''\}$ before crossing $C_1$ and this is impossible because, by Lemma~\ref{zeros},
\[
m(s'') > \frac{1}{M(s^{1/L^{1/8}})},
\]
whereas, by \eqref{vsmall} and Lemma~\ref{incr} part~(a),
\[
|f(z)| < \frac{1}{M(s^{1/L^{1/8}})},\; \mbox{ for } z \in C_1 \cap \overline{U_I}.
\]

We now suppose that at least $\tfrac14\log M(s^{1/L^{3/4}})$ of the first crossing points lie in $C(s^{1/L})$ and use Theorem~\ref{long-thin} to show that this leads to a contradiction. We first note that, provided we choose $R_1$ sufficiently large to ensure that
\begin{equation}\label{sL8}
 s^{1/L^{7/8} - 1/L} \geq 2(16)^2,
\end{equation}
we can apply Lemma~\ref{Be1} twice, first with $\delta = 1$, $\lambda = 1/16$ and $r=s^{1/L^{7/8}}$ and then with $\delta = 1$, $\lambda = 1/16$ and $r=s^{1/L^{7/8}}/32$, to deduce that there exist $s_1, s_2$ with
\begin{equation}\label{s1s2}
s^{1/L^{7/8}} \geq s_2 \geq 2s_1 \geq 2s^{1/L} \mbox{ such that } m(s_i) > \frac{1}{M(s^{1/L^{7/8}})}, \mbox{ for } i = 1,2.
\end{equation}

Thus, if more than $\tfrac14\log M(s^{1/L^{3/4}})$ of the first crossing points defined above lie in $C(s^{1/L})$, then there exists a quadrilateral $Q \subset U_I$, bounded by curves $\alpha, \alpha'$, which are arcs of the circles $C(s_1)$ and $C(s_2)$, and by curves $\beta, \beta'$ contained in two of the pre-images defined above, such that, for a suitable branch $F$ of $\log f$, we have that~$F$ is conformal in $Q$, that
\begin{equation}\label{beta}
\Im (F(z)) = 0, \mbox{ for } z \in \beta,
\end{equation}
\begin{equation}\label{beta1}
\Im (F(z)) \ge \tfrac12\pi\log M(s^{1/L^{3/4}}), \mbox{ for } z \in \beta',
\end{equation}
and, by \eqref{s1s2} and \eqref{fsmalla} together with Lemma~\ref{incr} part~(a), that
\begin{equation}\label{Q}
 -\log M(s^{1/L^{7/8}}) \leq \Re (F(z)) \leq \log M(s^{1/L}), \mbox{ for } z \in \alpha\cup\alpha'.
\end{equation}
We now consider the quadrilateral $\log Q$, where $\log$ is the principal branch of the logarithm, which has sides $\log \alpha$, $\log \beta$, $\log \alpha'$ and $\log \beta'$. This quadrilateral satisfies the hypotheses of Theorem~\ref{long-thin} with $\phi(z)=F(e^z)=\log f(e^z)$ and with
\[
a=\log 2,\quad b=\pi,\quad A=2\log M(s^{1/L^{7/8}}),\quad B=\tfrac12 \pi\log M(s^{1/L^{3/4}}),
\]
by \eqref{s1s2}, \eqref{beta}, \eqref{beta1} and \eqref{Q}. It follows from Theorem~\ref{long-thin} that
\[
\frac{\log 2}{\pi}\le \frac{2\log M(s^{1/L^{7/8}})}{\tfrac12 \pi\log M(s^{1/L^{3/4}})}=\frac{4\log M(s^{1/L^{7/8}})}{\pi\log M(s^{1/L^{3/4}})}\leq \frac{4}{\pi L^{1/8}},
\]
by~\eqref{Had}, since $s^{1/L^{7/8}} \geq R_0$ by~\eqref{sLR0}.
However, $L^{1/8} > 4/\log 2$, by~\eqref{K1cond}, so this is a contradiction.

To conclude, if $K_1$ and $R_1$ are chosen so as to satisfy Lemma~\ref{zeros}, and also so that~\eqref{sL8} holds, then at most $\tfrac14\log M(s^{1/L^{3/4}})$ pre-images of the positive real axis defined above leave $U_I$ by crossing $C(s^{1/L})$ and none leave by crossing $C_1$, so at least $\tfrac14\log M(s^{1/L^{3/4}})$ must leave by crossing $\gamma$. Thus there exist a continuum $\Gamma\subset \gamma\cap\overline A(s^{1/L},s)$ and $z_0, z'_0\in\Gamma$ such that
\[
\Delta\!\arg(f(\Gamma);z_0,z'_0) \geq \tfrac12\pi \log M(s^{1/L^{3/4}}) > \log M(s^{1/L})\,.
\]
This completes the proof of Theorem~\ref{main2}.
\end{proof}

\end{document}